\documentclass[11pt]{amsart}

\title{Power-Saving Bounds for Monic Minkowski Polynomials}

\usepackage{amsmath,amssymb,amsthm,amsfonts,mathtools,bm,xcolor}

\usepackage{graphicx}
\usepackage[all,arc,2cell]{xy}
\UseAllTwocells
\usepackage{tikz,tikz-cd}
\usetikzlibrary{trees,positioning,calc}
\usetikzlibrary{positioning}
\usepackage{amsmath}

\theoremstyle{plain}
\newtheorem{theorem}{Theorem}[section]

\newtheorem*{claim*}{Claim}

\newtheorem*{proposition*}{Proposition}

\newtheorem*{fact*}{Fact}

\newtheorem*{conjecture*}{Conjecture}

\newtheorem{lemma}[theorem]{Lemma}
\newtheorem*{lemma*}{Lemma}

\newtheorem*{question*}{Question}
\theoremstyle{definition}\newtheorem{remark}[theorem]{Remark}
\theoremstyle{definition}\newtheorem*{remark*}{Remark}
\theoremstyle{definition}
\theoremstyle{definition}\newtheorem*{definition*}{Definition}
\theoremstyle{definition}
\theoremstyle{definition}
\theoremstyle{definition}
\theoremstyle{definition}\newtheorem*{example*}{Example}

\renewcommand{\phi}{\varphi} 
\renewcommand{\epsilon}{\varepsilon} 

\def\Z{\mathbb{Z}} 
\def\R{\mathbb{R}} 

\begin{document}

\begin{abstract}
We prove that if $f\in \Z[x]$ is a monic polynomial of degree
$k\geq 2$, then there exists a constant $c>0$, depending only on $f$, and finite sets
$A\subset \mathbb R$ of arbitrarily large size such that
\[
|f(A)|\leq |A|^{k-c},
\]
where $f(A)$ is interpreted in the Minkowski sum-product sense. In
particular, taking $f(x)=x^2+x$, this gives a power-saving upper bound
for $AA+A$, answering a question raised by Roche-Newton, Ruzsa, Shen, and
Shkredov.
\end{abstract}

\author{Seamus Lavine}

\maketitle

\section{Introduction}
Given a subset $A$ of a ring $R$, the \textit{sumset} and \textit{product set} of $A$ are
\[
A+A=\{a+b:a,b\in A\}\quad\text{and}\quad AA=\{ab:a,b\in A\}.
\]
Erd\H{o}s and Szemer\'{e}di famously conjectured that, over the integers, one has
\[
\max\{|A+A|,|AA|\}\geq|A|^{2-o(1)}.
\]
The same conjecture was widely expected to hold over $\R$. However, a recent breakthrough of Bloom, Sawin, Schildkraut, and Zhelezov \cite{BloomSawinSchildkrautZhelezov2026} disproved analogues of the Erd\H{o}s--Szemer\'edi conjecture in several settings, including $\R$ (notably, their construction does not apply over the integers).

Motivated by the Erd\H{o}s-Szemer\'edi conjecture over $\R$, mathematicians have studied mixed sum-product expressions, or more generally \textit{Minkowski polynomials}, of finite sets of real numbers. If $f(x)=\sum_{j=0}^k a_jx^j\in\Z[x]$ and $A\subset \R$, we denote by $f(A)$ the Minkowski sum-product expression $f(A)=\sum_{j=0}^k a_jA^j$. One of the simplest such expressions is $AA+A$. Balog studied this set in \cite{Balog2011} and conjectured that 
\[
|AA+A|\geq |A|^2
\]
for every finite set $A$ of positive real numbers. Roche-Newton, Ruzsa, Shen, and Shkredov \cite{RocheNewtonRuzsaShenShkredov2019} disproved this conjecture by constructing sets $A$ of increasing size for which $|AA+A|=o(|A|^2)$. In the same paper, they asked whether the weaker bound 
\[
|AA+A|\geq |A|^{2-o(1)}
\]
always holds; a positive answer has since been described as a folklore conjecture \cite{AgrawalBloomPetridis2025}. Our main result shows that the answer is \textit{negative}, not only for $AA+A$, but for \textit{every} monic Minkowski polynomial $f\in\Z[x]$. 

\begin{theorem}
Let $f\in \Z[x]$ be a monic polynomial of degree $k\geq 2$. Interpreting $f(A)$ in the Minkowski sum-product sense, there exists a constant $c>0$, depending only on $f$, and finite sets $A\subset \R$ of arbitrarily large size such that
\[
|f(A)|\leq |A|^{k-c}.
\]
\end{theorem}

While the restriction to monic polynomials may seem arbitrary, Theorem 1.1 cannot hold for all non-monic polynomials \cite{Balog2011, BalogRocheNewton2015}. Balog and Roche-Newton showed that for $A\subset\R^{>0}$, 
\[
|4^{k-1}A^k|\geq|A|^k.
\]
Perhaps the simplest non-monic expression is $AA+AA$. A near-sharp lower bound of the form $|AA+AA|\geq|A|^{2-o(1)}$ was claimed in \cite{IosevichRocheNewtonRudnev2011}, but was later retracted \cite{IosevichRocheNewtonRudnev2018}. However, the authors nevertheless believe the claimed lower bound is true and relate it to other well-known problems in combinatorics.

\subsection*{Acknowledgments} We thank Alexander Razborov for helpful feedback throughout the development of this work and for comments on an earlier version of the manuscript. We also thank Kevin Lin for helpful discussions about algebraic number theory.

\section{Background}
The construction is heavily based on Bloom, Sawin, Schildkraut, and Zhelezov's (BSSZ) counterexample to the Erd\H{o}s-Szemer\'edi conjecture over $\R$. In particular, we use their high-dimensional Balog-Wooley construction, together with several of their counting estimates. The main new ingredient is described in the next section, where we sketch the proof of Theorem 1.1. In this section, we recall the high-dimensional Balog-Wooley example only briefly, and refer the reader to the original paper \cite{BloomSawinSchildkrautZhelezov2026} for further details.

The basic Balog-Wooley example consists of taking the product
\[
A=GP,
\]
where $G$ is a short geometric progression and $P$ is an interval. The multiplicative structure of $G$ controls the multiplicative doubling of $A$, while the additive structure of $P$ keeps $A+A$ inside a small interval. Both $|AA|$ and $|A+A|$ are smaller than the trivial $|A|^2$ by a logarithmic factor. 

To upgrade this logarithmic saving to a power saving, BSSZ replace the one-dimensional progression and interval by higher-dimensional analogues. Namely, they take $P$ to be a box in the lattice of algebraic integers $\mathcal O_K$ of a totally real number field $K$, and take $G$ to be a box in the logarithmic unit lattice of $K$. More precisely, if $K$ is a totally real number field of degree $d$ with embeddings $\sigma_1,\dots,\sigma_d:K\hookrightarrow\R$, define the additive and log-multiplicative boxes
\[
B^+(X)=\{\alpha\in\mathcal O_K:|\sigma_i(\alpha)|\leq X\text{ for all }1\leq i\leq d\},
\]
and
\[
B^\times(Y)=\{\alpha\in \mathcal O_K^\times:|\log|\sigma_i(\alpha)||\leq Y\text{ for all }1\leq i\leq d\}.
\]
We now record the lemmas of BSSZ we use below. The first estimates the size of the high-dimensional Balog-Wooley example (see the first inequality in Lemma 4.1 of \cite{BloomSawinSchildkrautZhelezov2026}).
\begin{lemma}[Bloom, Sawin, Schildkraut, Zhelezov] There exist absolute constants $C_0,C_1,\eta>0$ such that the following holds. Let $K$ be a totally real number field of degree $d\geq 2$ with discriminant $\Delta_K$. Let $0<\epsilon<\eta$, $Y\geq2$ be arbitrary, and $X\geq C_0^{1/\epsilon}$ be an integer. Set $G=B^\times(Y)$ and $P=X+B^+(\epsilon X)$. Then $A=GP\subset\mathcal O_K$ satisfies
\[
C_1X^dY^{d-1}\Delta_K^{-3/2}\leq|A|\leq X^dY^dC_1^{-d}.
\]
\end{lemma}
We also make use of their lattice counting bounds (see Lemmas 3.3 and 3.5 of \cite{BloomSawinSchildkrautZhelezov2026}), along with a theorem of Martinet on the existence of totally real number fields with bounded discriminant. 
\begin{lemma}[Bloom, Sawin, Schildkraut, Zhelezov] Let $K$ be a totally real number field of degree $d$. For all $X,Y\geq 1$,
\[
|B^+(X)|\leq(2X+1)^d\quad\text{and}\quad|B^\times(Y)|\leq 10(5Y+1)^{d-1}.
\]
\end{lemma}

\begin{theorem}[Martinet] There exists an absolute constant $C_2>0$ such that, for infinitely many $d$, there exist totally real number fields $K$ of degree $d$ with discriminant $\Delta_K\leq C_2^d$. 
\end{theorem}
\section{Sketch of the Proof}
Suppose $f(GP)=\sum_{j=1}^k a_j G^jP^j$; recall that this expression is to be interpreted in the Minkowski sum-product sense with $a_k=1$. Letting $\sigma=(\sigma_1,\dots,\sigma_d): K\hookrightarrow \R^d$, the embedded $\sigma(P)$ lives in a box of radius $X$ centered at $(X,\dots,X)$. Roughly, 
\[
\sigma(P)\subset [0,2X]^d.
\]
Similarly, $\sigma(P^j)$ lives at scale $X^j$ in each coordinate. Roughly,
\[
\sigma(P^j)\subset[0,O(X^j)]^d.
\]
Meanwhile, the logarithmic embeddings of elements of $G$ are bounded by $Y$, so multiplication by an element of $G^j$ stretches each coordinate by a factor of at most $e^{O(Y)}$. Thus, an element of $\sigma(G^jP^j)$ lives in a box of size roughly $e^{O(Y)}X^j$. Choosing $X$ to be sufficiently large compared to $e^{O(Y)}$, we have the separation of scales
\[
X^k\gg e^{O(Y)}X^{k-1}\gg e^{O(Y)}X^{k-2}\gg\cdots,
\]
where the notation $\gg,\ll$ is used to denote that the relevant inequality holds up to an absolute constant. Thus, a bounded enlargement of the $X^k$ box is large enough to contain all lower-order terms. In other words, for a fixed $u_k\in G^k$, $q_k\in P^k$, and $s=u_kq_k+\text{lower-order terms}\in f(GP)$, we have
\[
s\in u_kB^+(O(X^k))\quad\text{and}\quad f(GP)\subset\bigcup_{u_k\in G^k}u_k B^+(O(X^k)).
\]
Since multiplication by a unit is a bijection on $\mathcal O_K$, we have 
\[
|u_kB^+(O(X^k))|=|B^+(O(X^k))|
\]
and we are reduced to considering the sizes of $G^k$ and $B^+(O(X^k))$. 

We now turn to the details. 

\section{Proof of Theorem 1.1}
\begin{remark}
    In what follows, the constants $C_0,C_1,C_2$ are absolute and given by the lemmas. The constant $C_3$ depends on $f$, and subsequent constants depend on $f$ and the previous constants. 
\end{remark}

\begin{proof}[Proof of Theorem 1.1] For arbitrarily large $d$, let $K$ be a totally real number field of degree $d$ with embeddings $\sigma_1,\dots,\sigma_d:K\hookrightarrow\R$. By Theorem 2.3, we can take the discriminant $\Delta_K\leq C_2^d$ for an absolute $C_2>0$, independent of $d$. Let $Y$ be a sufficiently large constant, and we will choose a sufficiently large integer $X$ below; both $X$ and $Y$ are independent of $d$. Let 
\[
G=B^\times(Y)\quad\text{and}\quad P=X+B^+(\epsilon X),
\]
and set 
\[
\mathcal A=GP\subset \mathcal O_K.
\]
By Lemma 2.1, we have
\[
C_1X^dY^{d-1}\Delta_K^{-3/2}\leq |\mathcal A|\leq X^dY^dC_1^{-d}.
\tag{1}
\]
If $f(x)=\sum_{j=0}^k a_jx^j$ with $a_k=1$ and, without loss of generality, $a_0=0$, we can write
\[
f(\mathcal A)=\bigcup_{u_k\in G^k}\bigcup_{\substack{u_{j,i}\in G^j\\1\leq j\leq k-1,\;1\leq i\leq |a_j|}}\left(u_kP^k+\sum_{j=1}^{k-1}\operatorname{sgn}(a_j)\sum_{i=1}^{|a_j|}u_{j,i}P^j\right),
\]
using the convention $\operatorname{sgn}(0)=0$. Now fix $u_k\in G^k$ and $u_{j,i}\in G^j$, $(1\leq j\leq k-1,\;1\leq i\leq |a_j|)$, and let
\[
s=u_k q_k+\sum_{j=1}^{k-1}\operatorname{sgn}(a_j)\sum_{i=1}^{|a_j|}u_{j,i}q_{j,i},
\]
where $q_k\in P^k$ and $q_{j,i}\in P^j$. Then
\[
u_k^{-1}s=q_k+\sum_{j=1}^{k-1}\operatorname{sgn}(a_j)\sum_{i=1}^{|a_j|}u_k^{-1}u_{j,i}q_{j,i},
\]
and for every fixed embedding $\sigma_\ell$, we have
\[
|\sigma_\ell(u_k^{-1}s)|\leq|\sigma_\ell(q_k)|+\sum_{j=1}^{k-1}\sum_{i=1}^{|a_j|}|\sigma_\ell(u_k^{-1}u_{j,i})||\sigma_\ell(q_{j,i})|.
\]
Since $q_k,q_{j,i}$ lie in the corresponding $P$-boxes and
$u_k,u_{j,i}$ lie in the corresponding $G$-boxes, we have
\[
|\sigma_\ell(q_k)|\leq (1+\epsilon)^kX^k,\qquad
|\sigma_\ell(q_{j,i})|\leq (1+\epsilon)^jX^j,\qquad
|\sigma_\ell(u_k^{-1}u_{j,i})|\leq e^{2kY}.
\]
By choosing $X\geq (1+\epsilon)^{k-1}e^{2kY}$, we have
\[
|\sigma_\ell(u_k^{-1}s)|\leq(1+\epsilon)^kX^k+\sum_{j=1}^{k-1}|a_j|X^k.
\]
Assuming \(\epsilon\leq 1\), we have
\[u_kP^k+\sum_{j=1}^{k-1}\operatorname{sgn}(a_j)\sum_{i=1}^{|a_j|}u_{j,i}P^j\subset u_kB^+(C_3X^k),
\]
where $C_3=2^k+\sum_{j=1}^{k-1}|a_j|$. Thus
\[
\bigcup_{\substack{u_{j,i}\in G^j\\1\leq j\leq k-1,\;1\leq i\leq |a_j|}}\left(u_kP^k+\sum_{j=1}^{k-1}\operatorname{sgn}(a_j)\sum_{i=1}^{|a_j|}u_{j,i}P^j\right)
\subset u_kB^+(C_3X^k).
\]
Since multiplication by a unit is a bijection on $\mathcal O_K$, we have 
\[
|u_kB^+(C_3X^k)|=|B^+(C_3X^k)|\leq (2C_3X^k+1)^d\leq C_4^d X^{kd},
\]
for a constant $C_4>0$, where the last inequality follows by Lemma 2.2. Since $G^k\subset B^\times(kY)$, by Lemma 2.2 we have $|G^k|\leq(C_5Y)^{d-1}$ for a constant $C_5>0$. Thus, 
\[
|f(\mathcal A)|\leq|G^k|C_4^dX^{kd}\leq C_6^dY^{d-1}X^{kd}
\tag{2}
\]
for a constant $C_6>0$. Now combining $(1)$ and $(2)$, we have
\[
\frac{|f(\mathcal A)|}{|\mathcal A|^k}\leq \frac{Y^{k-1}}{C_1^k}\left(\frac{C_7}{Y^{k-1}}\right)^d
\]
for a constant $C_7>0$. By choosing $Y$ sufficiently large so that $C_7/Y^{k-1}<1$ and absorbing the extra $Y^{k-1}/C_1^k$ term into the exponent for sufficiently large $d$, we see that there is an $\alpha\in(C_7/Y^{k-1},1)$ with
\[
|f(\mathcal A)|\leq \alpha^d|\mathcal A|^k.
\]
To conclude, the upper bound in $(1)$ gives $|\mathcal A|\leq(XYC_1^{-1})^d$. Thus by choosing $c=\frac{-\log\alpha}{\log(XYC_1^{-1})}$, we have
\[
|f(\mathcal A)|\leq|\mathcal A|^{k-c}.
\]
Since $\sigma_1$ is an injective field homomorphism, taking $A=\sigma_1(\mathcal A)$ completes the proof.
\end{proof}

\end{document}